\documentclass[two side,reqno,12pt]{amsart}
\usepackage{amssymb,latexsym,xy,eucal,mathrsfs}
\usepackage{amssymb,amsmath,graphicx,colortbl,enumerate}
\usepackage{blindtext}
\usepackage{comment}
\usepackage{tikz}
\usetikzlibrary{decorations.markings}
\usepackage{pst-plot}
\usepackage{rotating}
\usepackage{pgfplots}
\usepackage{amsmath}
\usepackage{longtable,mathrsfs}
\usepackage[labelfont=bf,justification=raggedright,singlelinecheck=false]{caption}
\usepackage{caption}
\usepackage{subcaption}
\usepackage{framed}
\usepackage{float}
\numberwithin{equation}{section}
\newtheorem{lemma}{Lemma}[section]
\newtheorem{theorem}[lemma]{Theorem}

\newtheorem{corollary}[lemma]{Corollary}

\usepackage{graphicx}
\theoremstyle{definition}
\newtheorem{example}[lemma]{Example}

\newtheorem{definition}[lemma]{Definition}

\allowdisplaybreaks
\DeclareMathSizes{10}{9}{5}{5}   
\usepackage{mathtools}
\begin{document}
	\baselineskip 17truept
	\title{On solution of Fractional Diffusion Equation using Conformable Laplace Transform}
	\date{}
	\author{Somnath Sarate, Anil Khairnar and Krishnat Masalkar} 	
	
	\address{Somanth Sarate , Department~of~Engineering Science and Humanities, Marathwada~Mitra~ Mandal~College~of engineering, Karve~nagar, Pune, Pune-41004, India
		(M.S.)} \email{somnathsarate@mmcoe.edu.in}
	
	\address{Anil Khairnar, Department~of~Mathematics, Abasaheb~Garware~College~~Pune-411004~Pune, Pune-411007, India
		(M.S.)} \email{dr.khairnar@yahoo.com}
	\address{Krishnat Masalkar, Department~of~Mathematics, Abasaheb~Garware~College~~Pune-411004~Pune, Pune-411007, India
		(M.S.)} \email{krishnatmasalkar@gmail.com}
	\makeatletter
	\@namedef{subjclassname@2020}{%
		\textup{2020} Mathematics Subject Classification}
	\makeatother
	
	
	\maketitle 
	
	\maketitle 
	\begin{abstract}
The inversion theorem and convolution theorem of the conformable fractional Laplace transforms are developed. All the elementary properties of the classical Laplace transform are extended to the conformable fractional transform, and using these properties, we found analytical solutions to the initial-boundary value problems of the diffusion equation.
	\end{abstract}
	

	\section{Introduction}
Fractional derivatives, such as Riemann–Liouville and Caputo derivatives, can be expressed via the Gamma function using fractional integrals. Although these definitions are mathematically rigorous and widely used, they do not satisfy some fundamental properties of classical calculus, such as the standard product and chain rule, which makes the analytical treatment of fractional differential equations more complex. To overcome this limitation, Khalil et al. \cite{khalil2014} introduced the conformable fractional derivative, defined through a limit process similar to the classical derivative, preserving many important properties of ordinary calculus. Later, T. Abdeljawad \cite{abdeljawad2015} extended this concept further and developed the conformable fractional Laplace transform to systematically handle differential equations involving conformable derivatives. N. A. Khan et al. \cite{khan2018some} studied several properties and applications of this transform and demonstrated its effectiveness in solving fractional differential equations. Although it reduces to the classical Laplace transform via a suitable change of variables, its theory still needs systematic development through inversion and convolution formulas for broader applicability. The main aim of this paper is to derive inversion and convolution formulae for the conformable fractional Laplace transform and apply them to solve time-fractional diffusion equations in finite and semi-infinite domains.
	\section{Elementary properties }
Let $f:(-\infty,\infty)\rightarrow \mathbb{C}$ be a function
its classical derivative at $t$ is given by
\[f'(t)=\lim_{h\rightarrow 0}\frac{f(t+h)-f(t)}{h}.\]
Its natural extension in the fractional sense ie. for $\alpha \in [0,1]$  is given by Khalil et. al. in \cite{khalil2014}, and it is defined as
\begin{align}
T_{\alpha}(f(t))=\begin{cases}
\displaystyle 	\lim_{h\rightarrow 0}\frac{f(t+h|t|^{1-\alpha})-f(t)}{h} & t>0\\
\displaystyle 	\lim_{t\rightarrow 0^+} t^{1-\alpha}f'(t) & t=0
\end{cases}
\end{align}
For $\alpha\in [n,n+1)$ for any integer $n$, we have $[\alpha]=n$ and $\{\alpha\}=\alpha-[\alpha]$. Then the conformable fractional derivative is given by 
\[T_{\alpha}(f(t))=T_{\{\alpha\}}(f^{([\alpha])}(t)),\quad\text{where}\quad f^{(k)}(t)~\text{is classical }~k^{th}~\text{derivative of }~f.\]
Note that for $t>0$ if we put $ u=h|t|^{1-\alpha}$
then 
\begin{align*}
T_{\alpha}(f(t))&=\lim_{h\rightarrow 0}\frac{f(t+h|t|^{1-\alpha})-f(t)}{h}
=|t|^{1-\alpha}\lim_{u\rightarrow 0}\frac{f(t+u)-f(t)}{u}
=|t|^{1-\alpha} f'(t).
\end{align*}
Thus, $f'(t)$ exist if and only if $T_{\alpha}(f(t))$ exist for all $t>0$. But $T_{\alpha}(f(0))$ exist even though $f'(0)$ does not exist. For example if $f(t)=t^{\beta}$ then $f'(0)$ does not exist for $0\leq \beta<1$ but $T_{\alpha}(f(0))=0$.
The classical Laplace transform of a function $f:[0, \infty)\rightarrow \mathbb C$ is defined by the improper integral
\[ F(s)=\mathcal{L}(f(t))(s)=\int_{0}^{\infty} e^{-st}f(t)dt,\quad \operatorname{Re}(s) > a,\quad a\in \mathbb{R}\] 
If the function lies in the space of piecewise continuous (ie.$f$ has finite number of finite jump discontinuities only) and exponential order (ie. $|f(t)|<Me^{at}$ for some $M>0$ and $a\in \mathbb R$ ) the $\mathcal{L}(f(t))(s)$ exist for all $s$ with  $\operatorname{Re}(s) > a$. But there are many functions outside this space whose Laplace transform exist.
This definition of the classical Laplace transform can be naturally extended to the conformable fractional Laplace transform. The conformable fractional Laplace transform, introduced by Abdeljawad \cite{abdeljawad2015}, is defined as
\begin{align}
	F_{a}^{\alpha}(s)&=\mathcal{L}_{a}^{\alpha}(f(t))(s)=\int_{a}^{\infty} e^{-s\frac{(t-a)^{\alpha}}{\alpha}}(t-a)^{\alpha-1}f(t)dt,\\\text{where}& \quad a\in \mathbb R, f:[a,\infty]\rightarrow \mathbb C, \alpha\in[0,\infty).
\end{align}
Now we give an important characterization of the conformable fractional Laplace transform.
\begin{theorem}
\label{thm:2.1}
	Let $a\in \mathbb R$ and $\phi_a:[a,\infty)\rightarrow \mathbb R$ be the function defined by
	$\phi_a(t)=\frac{(t-a)^{\alpha}}{\alpha}.$
	Let $f:[a, \infty) \rightarrow \mathbb C$ be a function. Then $\mathcal{L}_{a}^{\alpha}(f(t))(s)$ exist if and only if $\mathcal{L}((f\circ\phi_a^{-1})(t))(s)$ exist on $[a,\infty)$ for all $s$ with $real(s)>b$ for some $b\in \mathbb R$. 
\end{theorem}
\begin{proof}
	If we use the change of variable $u=\phi_a(t)$ then we get\\
	$du=\phi_a'(t)dt=(t-a)^{\alpha-1}dt$ and hence
	\begin{align*}
	\mathcal{L}_{a}^{\alpha}(f(t))(s)&=\int_{a}^{\infty}e^{-s\frac{(t-a)^{\alpha}}{\alpha}}f(t)(t-a)^{\alpha-1}dt\\
	&=\int_{a}^{\infty}e^{-s\phi_a(t)}f(t)\phi_a'(t)dt\\
	&=\int_{0}^{\infty}e^{-su}f\circ\phi_a^{-1}(u)du\\
	&=\mathcal{L}((f\circ\phi_a^{-1})(t))(s).
	\end{align*}
Hence proof.
\end{proof}
There exist functions whose classical Laplace transform exists but whose conformable fractional Laplace transform may fail to exist. 
Consider function $f(t) = e^{t^{\beta}}$ with $\beta > 0$. Its classical Laplace transform exists if and only if $\beta \leq 1$.  Moreover, $
\mathcal{L}_{0}^{\alpha}(f(t)) = \mathcal{L}\left(f\left((\alpha u)^{1/\alpha}\right)\right)(s) = \frac{1}{\alpha} \mathcal{L}\left(e^{u^{\beta/\alpha}}\right)\left(\frac{s}{\alpha}\right),$
which exists if and only if $\beta \leq \alpha.$
\medskip
Note that if we have a function
$ f: [a, \infty) \to \mathbb{C}, $ then we can define a function
$ f_a(t) = f(t + a): [0, \infty) \to \mathbb{C}. $ Moreover,
\[
\mathcal{L}_a^\alpha(f(t))(s) = \int_a^\infty e^{-s \frac{(t-a)^\alpha}{\alpha}} f(t) (t-a)^{\alpha-1} \, dt
= \int_0^\infty e^{-s \frac{u^\alpha}{\alpha}} f(u + a) u^{\alpha-1} \, du = \mathcal{L}_0^\alpha (f_a(t))(s).
\]
Hence, to develop any property of $\mathcal{L}_a^{\alpha}$ is equivalent to developing the same property for $\mathcal{L}_0^{\alpha}$.
For this reason, throughout the remainder of this work, we investigate the properties of $\mathcal{L}_0^\alpha$ on the space of complex- valued functions defined on
$[0,\infty)$.

\medskip
\noindent
We now turn to the essential working properties of the conformable fractional Laplace transform, which are important in analytical computations as well as in solving conformable fractional differential equations.
\begin{theorem}
Let \(\mathcal{V}= \{ f : [0, \infty) \to \mathbb{C} : \mathcal{L}_0^\alpha(f(t)) \text{ exists} \} \) and \( f,g \in \mathcal{V} \). Suppose that
\(\mathcal{L}_0^\alpha \{ f(t) \} = F_\alpha(s) \). Then  
\begin{enumerate} \setlength{\itemsep}{4pt}
     \item{Linearity:}  
     $\mathcal{L}_{0}^{\alpha}$ is a linear map on the $\mathbb{C}$-vector space $V$,

    \item{Scaling Property:}
     For \( a > 0 \),
    \( \mathcal{L}_0^\alpha \{ f(at) \} = \frac{1}{a^\alpha} \, F\!\left( \frac{s}{a^\alpha} \right),\)

     \item{First Shifting Property:}
     For \( s > a \), \(\mathcal{L}_0^\alpha \left\{ e^{a\frac{t^\alpha}{\alpha}} f(t) \right\} = \mathcal{L}_{0}^{\alpha}(f(t))(s - a)\),

      \item{Second Shifting Property:}
      Suppose \( u_a(t) = 1 \), if \( t > a \), and \( u_a(t) = 0 \), if \( 0 \leq t \leq a \).
      Then, for \( a \geq 0 \), \(\mathcal{L}_0^\alpha \left\{ u_a(t) f\!\left( (t^\alpha - a^\alpha)^{1/\alpha} \right) \right\} = e^{-s\frac {a^\alpha}{\alpha}} \, \mathcal{L}_0^\alpha \{ f(t) \}(s),\)

      \item{Multiplication by $t^{\alpha}$:}
      \(\mathcal{L}_0^\alpha \{ t^\alpha f(t) \} = -\alpha \frac{d}{ds}\mathcal{L}_0^\alpha \{ f(t) \}(s) \),

      \item{Division by $t^{\alpha}$:}
      \(\mathcal{L}_0^\alpha \left\{ \frac{f(t)}{t^\alpha} \right\} = \frac{1}{\alpha} \int_s^\infty  \mathcal{L}_0^\alpha \{ f(t) \}(s) \, ds \).

      \item {Uniqueness:}
      If \(\mathcal{L}_0^\alpha \{ f(t) \}(s) = \mathcal{L}_0^\alpha \{ g(t) \}(s)\)
      for all \( s \) in some half-plane \( \operatorname{Re}(s) > a \),
      then \(f(t) = g(t), \quad \text{for all } t \geq 0.\)
    \end{enumerate}
\end{theorem}
\begin{proof}
 (1)  If \( f, g \in\mathcal{V} \), then the improper integrals \( \mathcal{L}_0^\alpha(f) \) and \( \mathcal{L}_0^\alpha(g) \) exist.
Hence \( \mathcal{L}_0^\alpha(f + g) \) and \( \mathcal{L}_0^\alpha(kf) \) exist for any \( k \in \mathbb{C} \).
Therefore \( \mathcal{V} \) is a \( \mathbb{C} \)-vector space.
By Theorem~\ref{thm:2.1},
\(\mathcal{L}_0^\alpha(f) = \mathcal{L}(f \circ \phi_0^{-1}),\)
where \( \phi_0(t) = \frac{t^\alpha}{\alpha} \).
Since \( \mathcal{L} \) is a linear operator, \( \mathcal{L}_0^\alpha \) is also a linear operator. \\
(2) Let \( g(t) = f(at) \) for \( a > 0 \).
Then 
\begin{align*}
\mathcal{L}_0^\alpha(f(at))(s)
          & =\mathcal{L}_0^\alpha(g(t))(s)\\
        &= \mathcal{L}(g(\phi_0^{-1}(t)))(s)\\
        & =\mathcal{L}(f(a\phi_0^{-1}(t)))(s)\\
        &=\mathcal{L}(f(\phi_0^{-1}(a^\alpha t)))(s)
        \quad\text{since}\quad  a\phi_0^{-1}(t) = \phi_0^{-1}(a^\alpha t) \\
         &= \frac{1}{a^\alpha} \mathcal{L}(f(\phi_0^{-1}(t)))\left( \frac{s}{a^\alpha} \right) \quad\text{since}\quad\mathcal{L}\{ h(at)\}=\frac{1}{a}\mathcal{L}\{h(t)\}\left( \frac{s}{a} \right)\\
     &=  \frac{1}{a^\alpha} \mathcal{L}_0^\alpha(f(t))\left( \frac{s}{a^\alpha} \right)
    \end{align*}
   (3) Let \(g(t) = e^{a \phi_0(t)} f(t),\) for all  \(t \geq 0,
   \) where \(\phi_0(t) = \frac{t^\alpha}{\alpha}.\)
   Then
     \begin{align*}
     \mathcal{L}_0^\alpha \left\{ e^{a\frac{t^\alpha}{\alpha}} f(t) \right\}
    &=\mathcal{L}_0^\alpha(g(t))(s)\\
    &= \mathcal{L}(g \circ \phi_0^{-1})(s)\\
   &= \mathcal{L}\left( e^{a \phi_0(\phi_0^{-1}(t))} f(\phi_0^{-1}(t)) \right)(s)\\
  &= \mathcal{L}\left( e^{a t} f(\phi_0^{-1}(t)) \right)(s)
  \quad\text{since}\quad \phi_0(\phi_0^{-1}(t)) = t\\
   &= \mathcal{L}\left( f(\phi_0^{-1}(t)) \right)(s - a)\quad\text{since}\quad \mathcal{L}\{e^{at}f(t)\}=F(s-a) \\
    & = \mathcal{L}_0^\alpha(f(t))(s - a).
  \end{align*}
(4) Let \(g(t) = u_a(t) f\!\left( (t^\alpha - a^\alpha)^{1/\alpha} \right)\), for  \(a \geq 0.\)
Then
\[
\mathcal{L}_0^\alpha(g(t))(s)
= \mathcal{L}(g(\phi_0^{-1}(t)))(s).
\]
Now,
\begin{align*}
g(\phi_0^{-1}(t))
= u_a(\phi_0^{-1}(t)) \, f\!\left( \left( (\phi_0^{-1}(t))^\alpha - a^\alpha \right)^{1/\alpha} \right)
= u_a(\phi_0^{-1}(t)) \, f\!\left( \phi_0^{-1}\!\left( t - \frac{a^\alpha}{\alpha} \right) \right).
\end{align*}
Using the second shifting property of the classical Laplace transform, we obtain
\begin{align*}
    \mathcal{L}(g(\phi_0^{-1}(t)))(s)
= e^{-s\frac{a^\alpha}{\alpha}} \, \mathcal{L}(f(\phi_0^{-1}(t)))(s)
= e^{-s\frac{a^\alpha}{\alpha}} \, \mathcal{L}_0^\alpha(f(t))(s).
\end{align*}
(5) Since \( \mathcal{L}_0^\alpha(f(t)) = \mathcal{L}(f \circ \phi_0^{-1}),\)
we have
\begin{align*}
   \mathcal{L}_0^\alpha(t^\alpha f(t))
     &= \mathcal{L}((t^\alpha f(t)) \circ \phi_0^{-1})\\
     &= \mathcal{L}((\phi_0^{-1}(t))^\alpha f(\phi_0^{-1}(t)))\\
     & = \mathcal{L}(\alpha t \, f(\phi_0^{-1}(t))) \quad \text{since}\quad(\phi_0^{-1}(t))^\alpha = \alpha t \\
    &= \alpha \, \mathcal{L}(t \, f(\phi_0^{-1}(t)))\\
    &= -\alpha \frac{d}{ds} \mathcal{L}(f(\phi_0^{-1}(t)))\quad    \text{because}\quad\mathcal{L}\{t \, h(t)\} = -\frac{d}{ds} \mathcal{L}\{h\mathcal{L}(t)\}      \\
    &= -\alpha \frac{d}{ds} \mathcal{L}_0^\alpha(f(t)).
\end{align*}
(6) Similarly, 
\begin{align*}
  \mathcal{L}_0^\alpha\left( \frac{f(t)}{t^\alpha} \right)
&= \mathcal{L}\left( \left( \frac{f(t)}{t^\alpha} \right) \circ \phi_0^{-1} \right)\\
&= \mathcal{L}\left( \frac{f(\phi_0^{-1}(t))}{(\phi_0^{-1}(t))^\alpha} \right)\\ 
&= \mathcal{L}\left( \frac{f(\phi_0^{-1}(t))}{\alpha t} \right)     \quad \text{since}\quad(\phi_0^{-1}(t))^\alpha = \alpha t \\
&= \frac{1}{\alpha} \, \mathcal{L}\left( \frac{f(\phi_0^{-1}(t))}{t} \right)\\
&= \frac{1}{\alpha} \int_s^\infty \mathcal{L}(f(\phi_0^{-1}(t)))(s) \, ds \quad\text{since}\quad\mathcal{L}\left\{\frac{h(t)}{t}\right\} = \int_s^\infty H(s) \, ds\\
&=\frac{1}{\alpha} \int_s^\infty  \mathcal{L}_0^\alpha \{ f(t) \}(s) \, ds 
\end{align*}
(7) Suppose
\(
\mathcal{L}_0^\alpha(f(t))(s) = \mathcal{L}_0^\alpha(g(t))(s)
\)
for all \( s \) in some half-plane \( \operatorname{Re}(s) > a\).
Then
\(
\mathcal{L}(f \circ \phi_0^{-1})(s) = \mathcal{L}(g \circ \phi_0^{-1})(s).
\)
By the uniqueness property of the classical Laplace transform,
\(
f \circ \phi_0^{-1} = g \circ \phi_0^{-1}.
\)
Hence
\(
f(t) = g(t), \quad \text{for all } t \geq 0.
\)
\end{proof}
\section{Conformable Laplace Transform for Derivatives and Integrals}
We come up with a conformable fractional Laplace transform of conformable derivatives and integrals, which is a critical analysis tool in the solution of conformable fractional differential equations.
\begin{theorem} \label{thm:2.1}
 Let \( f(t) \) be differentiable on \( [0, \infty) \) and suppose that
$
\mathcal{L}_0^\alpha(f(t))(s) = F_{\alpha}(s) \quad \text{where}\quad 0 < \alpha \leq 1.
$
Then the conformable fractional Laplace transform of the conformable derivative satisfies
\[
\mathcal{L}_{o}^{\alpha}\{T_{\beta}f(t)\}
= \left[e^{-s\frac{t^{\alpha}}{\alpha}}t^{\alpha-\beta}f(t)\right]_{0}^{\infty}+
s\mathcal{L}_{o}^{\alpha}\{t^{\alpha-\beta}f(t)\}
+
(\beta-\alpha)\mathcal{L}_{o}^{\alpha}\{t^{-\beta}f(t)\}
\]  
\end{theorem}
\begin{proof}
Since $f$ is differentiable,
$\displaystyle 
T_t^\beta f(t) = t^{1-\beta} \frac{df(t)}{dt}=t^{1-\beta}f'(t) \quad \text{for} \quad 0<\beta\leq 1.
$
Hence, using the conformable Laplace transform, we have
\[
\mathcal{L}_{o}^{\alpha}\{T_{\beta}f(t)\}
=
\int_{0}^{\infty}
e^{-s\frac{t^{\alpha}}{\alpha}}
t^{1-\beta}f'(t)t^{\alpha-1}\,dt.
\]
Combining the powers of $t$,
\[
=
\int_{0}^{\infty}
e^{-\frac{s}{\alpha}t^{\alpha}}
t^{\alpha-\beta}f'(t)\,dt.
\]
Applying integration by parts, let
\[
u=e^{-s\frac{t^{\alpha}}{\alpha}}t^{\alpha-\beta}, 
\quad dv=f'(t)dt.
\]
Then
\[
v=f(t).
\]
Hence,
\[
\mathcal{L}_{o}^{\alpha}\{T_{\beta}f(t)\}
=
\left[e^{-s\frac{t^{\alpha}}{\alpha}}t^{\alpha-\beta}f(t)\right]_{0}^{\infty}
-
\int_{0}^{\infty}
f(t)\frac{d}{dt}
\left(e^{-s\frac{t^{\alpha}}{\alpha}}t^{\alpha-\beta}\right)dt.
\]
The boundary term vanishes. Using the product rule,
\[
\frac{d}{dt}
\left(e^{-s\frac{t^{\alpha}}{\alpha}}t^{\alpha-\beta}\right)
=
-s t^{2\alpha-\beta-1}e^{-s\frac{t^{\alpha}}{\alpha}}
+
(\alpha-\beta)t^{\alpha-\beta-1}e^{-s\frac{t^{\alpha}}{\alpha}}.
\]
Substituting,
\[
\mathcal{L}_{o}^{\alpha}\{T_{\beta}f(t)\}
=
\int_{0}^{\infty}
f(t)e^{-s\frac{t^{\alpha}}{\alpha}}
\left[
s t^{2\alpha-\beta-1}+(\beta-\alpha)t^{\alpha-\beta-1}
\right]dt.
\]
Splitting the integrals,
\[
= \left[e^{-s\frac{t^{\alpha}}{\alpha}}t^{\alpha-\beta}f(t)\right]_{0}^{\infty}+
s\int_{0}^{\infty}
e^{-s\frac{t^{\alpha}}{\alpha}}f(t)t^{2\alpha-\beta-1}dt
+
(\beta-\alpha)
\int_{0}^{\infty}
e^{-s\frac{t^{\alpha}}{\alpha}}f(t)t^{\alpha-\beta-1}dt.
\]
By definition of the conformable Laplace transform,
\[
\int_{0}^{\infty}
e^{-s\frac{t^{\alpha}}{\alpha}}f(t)t^{2\alpha-\beta-1}dt
=
\mathcal{L}_{o}^{\alpha}\{t^{\alpha-\beta}f(t)\},
\]
\[
\int_{0}^{\infty}
e^{-s\frac{t^{\alpha}}{\alpha}}f(t)t^{\alpha-\beta-1}dt
=
\mathcal{L}_{o}^{\alpha}\{t^{-\beta}f(t)\}.
\]
Thus,
\[
\mathcal{L}_{o}^{\alpha}\{T_{\beta}f(t)\}
= \left[e^{-s\frac{t^{\alpha}}{\alpha}}t^{\alpha-\beta}f(t)\right]_{0}^{\infty}+
s\mathcal{L}_{o}^{\alpha}\{t^{\alpha-\beta}f(t)\}
+
(\beta-\alpha)\mathcal{L}_{o}^{\alpha}\{t^{-\beta}f(t)\}.
\]
This completes the proof.
\end{proof}
\begin{corollary}
Let $f(t)$ be differentiable on $[0,\infty)$. From Theorem~\ref{thm:2.1}, the conformable fractional Laplace transform of the conformable derivative yields several important special cases depending on the values of $\alpha$ and $\beta$.
\begin{enumerate}
\item If $0<\alpha<1$ and $0<\beta<1$, then
\[
\mathcal{L}_0^{\alpha}\{T_{\beta}f(t)\}
= s\mathcal{L}_0^{\alpha}\{t^{\alpha-\beta}f(t)\}
+(\beta-\alpha)\mathcal{L}_0^{\alpha}\{t^{-\beta}f(t)\},
\]
\item If $\beta=1$ and $0<\alpha<1$, then
\[
\mathcal{L}_0^{\alpha}\{T_{1}f(t)\}
= s\mathcal{L}_0^{\alpha}\{t^{\alpha-1}f(t)\}
+(1-\alpha)\mathcal{L}_0^{\alpha}\{t^{-1}f(t)\}.
\]
\item If $\alpha=\beta$, then
\[
\mathcal{L}_0^{\alpha}\{T_{\alpha}f(t)\}
= s\mathcal{L}_0^{\alpha}\{f(t)\}-f(0).
\]
\item If $\alpha=1$ and $\beta=1$, then the conformable Laplace transform reduces to the classical Laplace transform, and we obtain
\[
\mathcal{L}\{f'(t)\}=s\mathcal{L}\{f(t)\}-f(0),
\]
which coincides with the classical Laplace transform property for derivatives.
\end{enumerate}
\end{corollary}
\begin{theorem}
Let $f(t)$ be $n$-times differentiable on $[0,\infty)$ and let $0<\alpha,\beta\leq 1$. Then the conformable fractional Laplace transform of the $n$-th conformable derivative of order $\beta$ is
\begin{align*}
\mathcal{L}_0^{\alpha} \left\{ (T_\beta)^n f(t) \right\} = s^n \mathcal{L}_0^{\alpha} \left\{ t^{n(\alpha-\beta)} f(t) \right\} 
&+ \sum_{k=1}^n \binom{n}{k} (\beta-\alpha)^k s^{\,n-k} \mathcal{L}_0^{\alpha} \left\{ t^{(n-k)(\alpha-\beta)-k\beta} f(t) \right\} \\
+ \text{boundary terms}.
\end{align*}
where the boundary terms are
$
f(0), T_\beta f(0), \ldots, (T_\beta)^{\,n-1} f(0),
$ and the conformable Laplace transform is
\[
\mathcal{L}_0^{\alpha} \{ f(t) \} = \int_0^\infty e^{-s \frac{t^\alpha}{\alpha}} t^{\alpha-1} f(t) \, dt.
\]
\end{theorem}
\begin{proof}
    From Theorem~\ref{thm:2.1}, for the first conformable derivative, we have
    \[
\mathcal{L}_{o}^{\alpha}\{T_{\beta}f(t)\}
= \left[e^{-s\frac{t^{\alpha}}{\alpha}}t^{\alpha-\beta}f(t)\right]_{0}^{\infty}+
s\mathcal{L}_{o}^{\alpha}\{t^{\alpha-\beta}f(t)\}
+
(\beta-\alpha)\mathcal{L}_{o}^{\alpha}\{t^{-\beta}f(t)\}
\]  
Applying the conformable fractional Laplace transform again,
\[
\mathcal{L}_0^{\alpha}\{(T_{\beta})^{2}f(t)\}
=
\mathcal{L}_0^{\alpha}\{T_{\beta}(T_{\beta}f(t))\}.
\]
Using the first derivative formula with
\(
g(t)=T_{\beta}f(t),
\)
\[
\mathcal{L}_0^{\alpha}\{T_{\beta}g(t)\}
= \left[e^{-s\frac{t^{\alpha}}{\alpha}}t^{\alpha-\beta}g(t)\right]_{0}^{\infty}+
s\mathcal{L}_0^{\alpha}\{t^{\alpha-\beta}g(t)\}
+
(\beta-\alpha)\mathcal{L}_0^{\alpha}\{t^{-\beta}g(t)\}.
\]
Hence,
\[
\mathcal{L}_0^{\alpha}\{(T_{\beta})^{2}f(t)\}
= \left[e^{-s\frac{t^{\alpha}}{\alpha}}t^{\alpha-\beta}T_{\beta}f(t)\}\right]_{0}^{\infty}+
s\mathcal{L}_0^{\alpha}\{t^{\alpha-\beta}T_{\beta}f(t)\}
+
(\beta-\alpha)\mathcal{L}_0^{\alpha}\{t^{-\beta}T_{\beta}f(t)\}.
\]
Using the definition of the conformable derivative
\(
T_{\beta}f(t)=t^{1-\beta}f'(t),
\)
we obtain
\[
t^{\alpha-\beta}T_{\beta}f(t)
=t^{\alpha-\beta}\,t^{1-\beta}f'(t)=t^{\alpha+1-2\beta}f'(t),
\quad\text{and}\quad t^{-\beta}T_{\beta}f(t) = t^{1-2\beta}f'(t).
\]
Thus,
\begin{align*}
\mathcal{L}_0^{\alpha}\{(T_{\beta})^{2}f(t)\}
&= \left[e^{-s\frac{t^{\alpha}}{\alpha}}t^{\alpha-\beta}t^{\alpha-\beta}T_{\beta}f(t)\}\right]_{0}^{\infty}\\
& +
s\mathcal{L}_0^{\alpha}\{t^{\alpha+1-2\beta}f'(t)\}\\
&+
(\beta-\alpha)\mathcal{L}_0^{\alpha}\{t^{1-2\beta}f'(t)\}.
\end{align*}
Applying integration by parts to the above transforms, we obtain
\[
\mathcal{L}_0^{\alpha}\{t^{\alpha+1-2\beta}f'(t)\}
=
\left[e^{-\frac{s t^{\alpha}}{\alpha}} t^{2\alpha-2\beta}f(t)\right]_{0}^{\infty}
+s\mathcal{L}_0^{\alpha}\{t^{2(\alpha-\beta)}f(t)\}
+(\beta-\alpha)\mathcal{L}_0^{\alpha}\{t^{\alpha-2\beta}f(t)\},
\]
and
\[
\mathcal{L}_0^{\alpha}\{t^{1-2\beta}f'(t)\}
=
\left[e^{-\frac{s t^{\alpha}}{\alpha}} t^{\alpha-2\beta}f(t)\right]_{0}^{\infty}
+s\mathcal{L}_0^{\alpha}\{t^{\alpha-2\beta}f(t)\}
+(\beta-\alpha)\mathcal{L}_0^{\alpha}\{t^{-2\beta}f(t)\}.
\]
Substituting these expressions back, we obtain
\begin{align*}
\mathcal{L}_0^{\alpha}\{(T_{\beta})^{2}f(t)\}
&= \left[e^{-s\frac{t^{\alpha}}{\alpha}}t^{\alpha-\beta}t^{\alpha-\beta}T_{\beta}f(t)\}\right]_{0}^{\infty}\\
& + s\left[e^{-\frac{s t^{\alpha}}{\alpha}} t^{2\alpha-2\beta}f(t)\right]_{0}^{\infty}
+s^2\mathcal{L}_0^{\alpha}\{t^{2(\alpha-\beta)}f(t)\}
+s(\beta-\alpha)\mathcal{L}_0^{\alpha}\{t^{\alpha-2\beta}f(t)\}\\
&+ (\beta-\alpha)\left[e^{-\frac{s t^{\alpha}}{\alpha}} t^{\alpha-2\beta}f(t)\right]_{0}^{\infty}
+s(\beta-\alpha)\mathcal{L}_0^{\alpha}\{t^{\alpha-2\beta}f(t)\}\\
&+(\beta-\alpha)^2\mathcal{L}_0^{\alpha}\{t^{-2\beta}f(t)\}.
\end{align*}
\begin{align*}
\mathcal{L}_0^{\alpha}\{(T_{\beta})^{2}f(t)\}
&=s^2\mathcal{L}_0^{\alpha}\{t^{2(\alpha-\beta)}f(t)\}
+s(\beta-\alpha)\mathcal{L}_0^{\alpha}\{t^{\alpha-2\beta}f(t)
+s(\beta-\alpha)\mathcal{L}_0^{\alpha}\{t^{\alpha-2\beta}f(t)\}\\
&+(\beta-\alpha)^2\mathcal{L}_0^{\alpha}\{t^{-2\beta}f(t)\}
+ \text{boundary terms}
\end{align*}
where boundary terms are
\begin{align*}
 (\beta-\alpha)\left[e^{-\frac{s t^{\alpha}}{\alpha}} t^{\alpha-2\beta}f(t)\right]_{0}^{\infty},\quad
 s\left[e^{-\frac{s t^{\alpha}}{\alpha}} t^{2\alpha-2\beta}f(t)\right]_{0}^{\infty},
 \quad\left[e^{-s\frac{t^{\alpha}}{\alpha}}t^{\alpha-\beta}t^{\alpha-\beta}T_{\beta}f(t)\}\right]_{0}^{\infty}
\end{align*}
Since the exponential term vanishes as $t \to \infty$, the boundary contributions reduce to the initial values involving
$
f(0)$, $T_{\beta}f(0).$
 Therefore, the conformable fractional Laplace transform of the second conformable derivative is given by
\[
\mathcal{L}_0^{\alpha}\{(T_{\beta})^{2}f(t)\}
=
s^{2}\mathcal{L}_0^{\alpha}\{t^{2(\alpha-\beta)}f(t)\}
+2(\beta-\alpha)s\mathcal{L}_0^{\alpha}\{t^{\alpha-2\beta}f(t)\}
+(\beta-\alpha)^{2}\mathcal{L}_0^{\alpha}\{t^{-2\beta}f(t)\},
\]
together with boundary terms involving $f(0)$ and $T_{\beta}f(0)$ where $0<\alpha,\beta\le1$.\\
Continuing this procedure, each application of the conformable fractional Laplace transform introduces an additional factor of $s$ and additional terms involving powers of $(\beta-\alpha)$. The coefficients follow the binomial pattern arising from repeated application of the transform.
Thus, after $n$ iterations we obtain
\begin{align*}
\mathcal{L}_0^{\alpha} \left\{ (T_\beta)^n f(t) \right\} = s^n \mathcal{L}_0^{\alpha} \left\{ t^{n(\alpha-\beta)} f(t) \right\} 
&+ \sum_{k=1}^n \binom{n}{k} (\beta-\alpha)^k s^{\,n-k} \mathcal{L}_0^{\alpha} \left\{ t^{(n-k)(\alpha-\beta)-k\beta} f(t) \right\} \\
+ \text{boundary terms}.
\end{align*}
where the boundary terms are
$
f(0), T_\beta f(0), \ldots, (T_\beta)^{\,n-1} f(0).
$\\
Hence proved. 
\end{proof}
\begin{theorem}\label{thm:3.4}
Let the conformable fractional integral of a function
$f(t)$ be defined by
$
g(t)=\int_{0}^{t}f(p)p^{\beta-1}dp,\quad 0<\beta\le 1,
$
and suppose that the conformable fractional Laplace transform of 
$f(t)$ exists and is given by
$
\mathcal{L}_0^{\alpha}\{f(t)\}=F_{\alpha}(s)\quad0<\alpha\le 1,
$ then conformable fractional Laplace transform of $g(t)$ is
\[
\mathcal{L}_{o}^{\alpha}\{g(t)\}=\frac{1}{s}\mathcal{L}_{o}^{\alpha}\{t^{\alpha-\beta}f(t)\}.
\]
\end{theorem}
\begin{proof}
  By the definition of the conformable fractional Laplace transform,
\[
\mathcal{L}_{o}^{\alpha}\{g(t)\}
=
\int_{0}^{\infty} e^{-s\frac{t^{\alpha}}{\alpha}} g(t)t^{\alpha-1}\,dt.
\]
Substituting the definition of $g(t)$, we obtain
\[
\mathcal{L}_{o}^{\alpha}\{g(t)\}
=
\int_{0}^{\infty} e^{-s\frac{t^{\alpha}}{\alpha}}
\left(\int_{0}^{t} f(p)p^{\beta-1}\,dp\right)t^{\alpha-1}\,dt.
\]
Interchanging the order of integration, we get
\[
\mathcal{L}_{o}^{\alpha}\{g(t)\}
=
\int_{0}^{\infty} f(p)p^{\beta-1}
\left[
\int_{p}^{\infty} e^{-s\frac{t^{\alpha}}{\alpha}}t^{\alpha-1}\,dt
\right]dp.
\]
Let
\[
u=\frac{t^{\alpha}}{\alpha}, \qquad du=t^{\alpha-1}dt.
\]
Then
\[
\mathcal{L}_{o}^{\alpha}\{g(t)\}
=
\int_{0}^{\infty} f(p)p^{\beta-1}
\left[
\int_{\frac{p^{\alpha}}{\alpha}}^{\infty} e^{-su}\,du
\right]dp.
\]
Evaluating the inner integral,
\[
\int_{\frac{p^{\alpha}}{\alpha}}^{\infty} e^{-su}\,du
=\frac{1}{s}e^{-s\frac{p^{\alpha}}{\alpha}}.
\]
Substituting this result, we obtain
\[
\mathcal{L}_{o}^{\alpha}\{g(t)\}
= \frac{1}{s}\int_{0}^{\infty} e^{-s\frac{p^{\alpha}}{\alpha}} f(p)p^{\beta-1}\,dp.
\]
Now write
\[
p^{\beta-1}=p^{\alpha-1}p^{\beta-\alpha}.
\]
Hence,
\[
\mathcal{L}_{o}^{\alpha}\{g(t)\}
=
\frac{1}{s}\int_{0}^{\infty} e^{-s\frac{p^{\alpha}}{\alpha}} f(p)p^{\alpha-1}p^{\beta-\alpha}\,dp.
\]
By the definition of the conformable Laplace transform,
\[
\int_{0}^{\infty} e^{-s\frac{p^{\alpha}}{\alpha}} f(p)p^{\alpha-1}p^{\beta-\alpha}\,dp
= \mathcal{L}_{o}^{\alpha}\{t^{\alpha-\beta}f(t)\}.
\]
Therefore,
\[
\mathcal{L}_{o}^{\alpha}\{g(t)\}=\frac{1}{s}\mathcal{L}_{o}^{\alpha}\{t^{\alpha-\beta}f(t)\}.
\]
This completes the proof.
\end{proof}
\begin{corollary}

Let the conformable fractional integral of a function $f(t)$ be defined by\ \
$
g(t)=\int_{0}^{t} f(p)p^{\beta-1}\,dp, \ \ 0<\beta\le1,
$
and suppose that the conformable fractional Laplace transform of $f(t)$ exists and is given by\ \
$
\mathcal{L}_{o}^{\alpha}\{f(t)\}=F_{\alpha}(s), \ \ 0<\alpha\le1.
$
Then the result obtained in Theorem~\ref{thm:3.4} yields several important special cases depending on the values of $\alpha$ and $\beta$.
\begin{itemize}
\item[(i)] If $0<\alpha<1$ and $0<\beta<1$, then
\[
\mathcal{L}_{o}^{\alpha}\left\{\int_{0}^{t} f(p)p^{\beta-1}\,dp\right\}
=\frac{1}{s}\mathcal{L}_{o}^{\alpha}\{t^{\alpha-\beta}f(t)\}.
\]
\item[(ii)] If $0<\alpha<1$ and $\beta=1$, then
\[
\mathcal{L}_{o}^{\alpha}\left\{\int_{0}^{t} f(p)\,dp\right\}
=
\frac{1}{s}\mathcal{L}_{o}^{\alpha}\{t^{\alpha-1}f(t)\}.
\]
\item[(iii)] If $\alpha=\beta$, then
\[
\mathcal{L}_{o}^{\alpha}\left\{\int_{0}^{t} f(p)p^{\alpha-1}\,dp\right\}
=
\frac{1}{s}\mathcal{L}_{o}^{\alpha}\{f(t)\}.
\]
Thus, the obtained result reduces to a simpler form analogous to the classical Laplace transform property of integration.
\item[(iv)] If $\alpha=1$ and $\beta=1$, then the conformable Laplace transform reduces to the classical Laplace transform, and we obtain
\[
\mathcal{L}\left\{\int_{0}^{t} f(p)\,dp\right\}
=\frac{F(s)}{s}.
\]
\end{itemize}
\end{corollary}

And now that we have obtained the conformable fractional Laplace transform relations for both derivatives and integrals, we proceed to consider its limiting properties. Specifically, the Initial and Final Value Theorems help in explaining how the function behaves around the origin and as it approaches infinity.
\begin{theorem}[Initial Value Theorem]
   Let the conformable fractional Laplace transform of 
$f(t)$ be given by 
\( \mathcal{L}_0^{\alpha}\{f(t)\} = F_{\alpha}(s), \quad 0 < \alpha \leq 1.
\)
Then,
\[
\lim_{t \to a^+} f(t) = \lim_{s \to \infty} s F_{\alpha}(s).
\]
\end{theorem}
\begin{proof}
Let $\phi_0(t)=\frac{t^{\alpha}}{\alpha}$ and $h(t)=f(\phi_0^{-1}(t))$. Then 
\[ H(s)=\mathcal{L}(h(t))=\mathcal L_0^{\alpha}(f(t))(s)=F_{\alpha}(s)\]
Hence, by the initial value theorem of $\mathcal L$,
\[
\lim_{t \to 0^+} f(t)=\lim_{t \to 0^+} h(t) = \lim_{s \to \infty} s H(s)=\lim_{s \to \infty} s F_{\alpha}(s),
\]
Note that since \( t \to 0^+ \), \(\phi_0^{-1}(t) \to 0,\) and so
$\displaystyle
\lim_{t \to 0^+} h(t) = \lim_{t \to 0^+} f(t).
$
\end{proof}
\begin{theorem}[Final Value Theorem]
   Let the conformable fractional Laplace transform of a function 
$
f(t)$ be
\(
\mathcal{L}_0^{\alpha}\{f(t)\} = F_{\alpha}(s), \quad 0 < \alpha \leq 1.
\)
Then,
\[
\lim_{t \to \infty} f(t) = \lim_{s \to 0} s F_{\alpha}(s).
\]
\end{theorem}
\begin{proof}
Let $\phi_0(t)=\frac{t^{\alpha}}{\alpha}$ and $h(t)=f(\phi_0^{-1}(t))$. Then 
\[ H(s)=\mathcal{L}(h(t))=\mathcal L_0^{\alpha}(f(t))(s)=F_{\alpha}(s)\]
Hence, by the final value theorem of $\mathcal L$,
\[
\lim_{t \to \infty} f(t)=\lim_{t \to \infty} h(t) = \lim_{s \to 0} s H(s)=\lim_{s \to 0} s F_{\alpha}(s),
\]
Note that since \( t \to \infty \), \(\phi_0^{-1}(t) \to \infty,\) and so
$\displaystyle
\lim_{t \to \infty} h(t) = \lim_{t \to \infty} f(t).
$
\end{proof}
\section{Inversion Theorem }
The conformable Laplace transform is interesting in that it has an inversion theorem, giving it the capability to reconstruct the original function in terms of its transform. Consequently, the approach is fully developed and is reliable when used in solving conformable fractional differential equations.
\begin{theorem} Let $f:[0,\infty)\rightarrow \mathbb C$ be a function whose classical Laplace transform $F(s)=\mathcal L(f(t))(s)$ exists for all $s\in \mathbb C$ such that $\mathcal R(s)>a$, for some $a\in\mathbb R$.
  Let its  conformable fractional Laplace transform  be
$
F_{\alpha}(s) = \mathcal{L}_{\alpha}\{f(t)\}(s),
$ then the inverse conformable fractional Laplace transform is expressed as
\[
f(t) = \mathcal{L}_{\alpha}^{-1}\{F_{\alpha}(s)\} = \mathcal{L}^{-1}\{F_{\alpha}(s)\}\bigg|_{u=\phi_0(t)},\quad\phi_0(t) = \frac{t^\alpha}{\alpha}
\]
where \(\mathcal{L}^{-1}\) denotes the classical inverse Laplace transform.
\end{theorem}
\begin{proof}
Let \[\phi_0(t)=\frac{t^{\alpha}}{\alpha}, \quad t\in[0, \infty), \quad 0<\alpha\leq 1.\]
Define \[ h(u) = f(\phi_0^{-1}(u)) \quad \text{for} \quad u\in[0,\infty).\]
Hence \[f(t)=h(\phi_0(t))\quad \text{for} \quad u\in[0,\infty).\]
Then 
\[F_\alpha(s) =\mathcal{L}_\alpha \{ f(t) \}(s) =\mathcal{L}\{ f(\phi_0^{-1}(t))= \mathcal{L}\{ h(t)\}(s).\]
From inversion theorem of $\mathcal L$,
\[
h(t) = \mathcal{L}^{-1}\{ F_\alpha(s) \}(t).
\]
Hence 
\[
f(t) = h(\phi_0(t)) = \left. \mathcal{L}^{-1}\{ F_\alpha(s) \} \right|_{u = \frac{t^\alpha}{\alpha}}.
\] 
where \(\mathcal{L}^{-1}\) denotes the classical inverse Laplace transform.
\end{proof}
The following examples illustrate the fractional inverse Laplace transform.\\
$
\displaystyle \mathcal{L}^{-1}_{\alpha}\!\left(\frac{1}{s}\right) = 1,
\quad\displaystyle \mathcal{L}^{-1}_{\alpha}\!\left(\frac{1}{s-\lambda}\right) = e^{\lambda \frac{ t^{\alpha}}{\alpha}},\quad
\displaystyle\mathcal{L}^{-1}_{\alpha}\!\left(\frac{1}{s+\lambda}\right) = e^{-\lambda \frac{ t^{\alpha}}{\alpha}},\quad 
\displaystyle\mathcal{L}^{-1}_{\alpha}\!\left(\frac{1}{s^2+\lambda^2}\right) = \sin\!\left( \lambda \frac{ t^\alpha}{\alpha} \right),
\quad \displaystyle\mathcal{L}^{-1}_{\alpha}\!\left(\frac{\lambda}{s^2+\lambda^2}\right) = \cos\!\left( \lambda \frac{ t^\alpha}{\alpha} \right).  
$
\begin{theorem}[Complex Inversion Formula for Conformable Laplace Transform] 
\label{thm:4.2}
    Let 
$
f: [0,\infty) \to \mathbb{R}
$
be piecewise continuous and of conformable exponential order \(a\).
Suppose
$
\mathcal{L}_0^\alpha \{ f(t) \} = F_{\alpha}(s) = \int_0^\infty e^{-s\frac{t^\alpha}{\alpha}} f(t) \, t^{\alpha-1} \, dt, \quad \operatorname{Re}(s) > a.
$
Then for every \(c > a\) and \(t > 0\),
\[
f(t) = \frac{1}{2\pi i} \int_{c-i\infty}^{c+i\infty} e^{s\frac{t^\alpha}{\alpha}} F_{\alpha}(s) \, ds.
\]
\end{theorem}
\begin{proof}
Let \[\phi_0(t)=\frac{t^{\alpha}}{\alpha}, \quad t\in[0, \infty), \quad 0<\alpha\leq 1.\]
Define \[ h(u) = f(\phi_0^{-1}(u)) \quad \text{for} \quad u\in[0,\infty).\]
Hence \[f(t)=h(\phi_0(t))\quad \text{for} \quad u\in[0,\infty).\]
Then 
\[F_\alpha(s) =\mathcal{L}_\alpha \{ f(t) \}(s) =\mathcal{L}\{ f(\phi_0^{-1}(t))= \mathcal{L}\{ h(t)\}(s).\]
From the complex inversion theorem of \(\mathcal{L}\),
\[
h(t) = \frac{1}{2\pi i} \int_{c-i\infty}^{c+i\infty} e^{su} F_\alpha(s) \, ds.
\] 
Hence 
\[
f(t) = h(\phi_0(t)) = \frac{1}{2\pi i} \int_{c-i\infty}^{c+i\infty} e^{s\frac{t^\alpha}{\alpha}} F_{\alpha}(s) \, ds.
\] 
Proof completed.  
\end{proof}
\begin{theorem}[Residue Form of the Conformable Inversion Formula]
    Let 
$
f: [0,\infty) \to \mathbb{R}
$
be piecewise continuous and of conformable exponential order \(a\). Suppose
$
\mathcal{L}_\alpha \{ f(t) \} = F_{\alpha}(s) $ in addition, that \(F_{\alpha}(s)\) admits a meromorphic extension to the half-plane \(\operatorname{Re}(s) > a\), also \(F_{\alpha}(s)\) satisfies the standard growth conditions required for contour deformation.
Then, for \(t > 0\),
\[
f(t) = \sum_{k} \operatorname{Res}\!\left( e^{s\frac{t^\alpha}{\alpha}} F_{\alpha}(s), \, s = z_k \right),
\]
where \(z_k\) are the poles of \(F_{\alpha}(s)\) in the half-plane \(\operatorname{Re}(s) > a\).
\end{theorem}
\begin{proof}
    From Theorem~\ref{thm:4.2}, we have the inversion representation
\[
f(t) = \frac{1}{2\pi i} \int_{c-i\infty}^{c+i\infty} e^{s\frac{t^\alpha}{\alpha}} F_{\alpha}(s) \, ds, \quad c > a.
\]
Under the additional assumption that \(F_{\alpha}(s)\) is meromorphic and of appropriate growth bounds, it is possible to take the contour of integration to be closed to the left in the complex plane. Under the classical residue theorem, the Bromwich integral value is equal to the sum of the residues of the integrand within the contour. Hence,
\[
f(t) = \sum_{k} \operatorname{Res}\!\left( e^{s\frac{t^\alpha}{\alpha}} F_{\alpha}(s), \, s = z_k \right).
\]
\end{proof}
\begin{example}
    If 
    $
F_{\alpha}(s) = \frac{1}{s - a},
$
then \(F_{\alpha}(s)\) has a simple pole at \(s = a\). The residue is
\[
\operatorname{Res}\!\left( e^{s\frac{t^\alpha}{\alpha}} F_{\alpha}(s), \, s = a \right) = e^{a\frac{t^\alpha}{\alpha.}} 
, \quad\Rightarrow\quad f(t) = e^{a\frac{t^\alpha}{\alpha}}. 
\]
\end{example}
\section{Convolution }
We prove the convolution theorem that corresponds with the conformable Laplace transform, which is actually a convenient device to use when manipulating products in the transform domain.
\begin{definition}
     Let \( f, g : [0, \infty) \to \mathbb{R} \) be functions. The comfortable fractional convolution of 
\( f \) and \( g \), denoted by \( *_{\alpha} \), is defined as
\[
(f *_{\alpha} g)(t) = \int_{0}^{t} f(p) \; g\!\left( \big( t^{\alpha} - p^{\alpha}\big)^{1/\alpha} \right) p^{\alpha-1} \, dp \tag{5.1}
\]
where \(*_{\alpha}\) denotes the fractional convolution.
\end{definition}
\begin{theorem}
The fractional convolution operator \( *_{\alpha} \) satisfies the following properties:
\begin{enumerate}
    \item 
    $ f *_{\alpha} g = g *_{\alpha} f  \quad\text{(commutativity)}; $
     \item
    $
    f *_{\alpha} (g *_{\alpha} h) = (f *_{\alpha} g) *_{\alpha} h\quad\text{(associativity)};
    $
    \item 
    $
    f *_{\alpha} (g + h) = (f *_{\alpha} g) + (f *_{\alpha} h) \quad\text{(distributivity)};
    $
     \item 
   $c(f *_{\alpha} g) = (cf) *_{\alpha} g = f *_{\alpha} (cg), \quad c \in \mathbb{R} \quad\text{(scalar property)}.$
\end{enumerate}
\end{theorem}
\begin{proof}
(1) Start from the definition of the convolution
\[
(g *_{\alpha} f)(t) = \int_{0}^{t} g(p) \; f\!\left(\big( t^{\alpha} - p^{\alpha}\big)^{1/\alpha} \right) p^{\alpha-1} \, dp 
\]
Using substitution
$
y = \big( t^{\alpha} - p^{\alpha}\big)^{1/\alpha} 
$
that is 
$
y^{\alpha} = t^{\alpha} - p^{\alpha}
$, we obtain
\begin{align*}
    (g *_{\alpha} f)(t) 
    &= \int_{0}^{t} f(y) \; g\!\left(\big( t^{\alpha} - y^{\alpha}\big)^{1/\alpha} \right) y^{\alpha-1} \, dy\\
    &=(f *_{\alpha} g)(t).
\end{align*}
(2) Starting with the left-hand side:
\begin{align*}
    \big(f *_\alpha (g *_\alpha h)\big)(t) 
    &= \int_0^t f(s) (g *_\alpha h)\big((t^\alpha - s^\alpha)^{1/\alpha}\big) s^{\alpha-1} \, ds\\
    &= \int_0^t f(s) \left[ \int_0^{(t^\alpha - s^\alpha)^{1/\alpha}} g(p) \, h\Big(\big((t^\alpha - s^\alpha) - p^\alpha\big)^{1/\alpha}\Big) p^{\alpha-1} \, dp \right] s^{\alpha-1} \, ds\\
    &=\int_0^t \int_0^t f(s)g(p)h((t^\alpha-s^\alpha-p^\alpha)^{1/\alpha})s^{\alpha-1}p^{\alpha-1}\,dp\,ds\\
    &=\int_0^t (f*_\alpha g)(u)h((t^\alpha-u^\alpha)^{1/\alpha})u^{\alpha-1}\,du\\
    &=\big((f *_\alpha g) *_\alpha h\big)(t).
    \end{align*}
(3) Follows immediately from the linearity of the integral.
\\(4) Starting with the left-hand side
\begin{align*}
   \big(f*_\alpha (g+h)\big)(t)
   &=\int_0^t f(p)(g+h)\!\left( \big( t^{\alpha} - p^{\alpha}\big)^{1/\alpha} \right) p^{\alpha-1} \, dp\\
   &= (f *_{\alpha} g) + (f *_{\alpha} h).
\end{align*}
Hence, the fractional convolution operator \( *_\alpha \) satisfies the commutativity, associativity, and distributivity properties, and therefore preserves the fundamental algebraic structure of the classical convolution.
\end{proof}
To examine the behavior of the convolution operator, we initially define some useful properties, such as the boundedness in $L^1_\alpha$ and Young’s inequality in Young Inequality in $L^n_\alpha$.
The classical Lebesgue space $L^1(0,\infty)$ consists of all measurable functions $f:(0,\infty)\to \mathbb{R}$ such that
\[
\int_0^\infty |f(t)| \, dt < \infty.
\]
The norm is defined by
\[
\|f\|_{L^1} = \int_0^\infty |f(t)| \, dt.
\]
More generally, for $1 \leq n < \infty$, the space $L^p(0,\infty)$ consists of all measurable functions $f:(0,\infty)\to \mathbb{R}$ satisfying
\[
\int_0^\infty |f(t)|^n \, dt < \infty.
\]
The associated norm is
\[
\|f\|_{L^n} = \left( \int_0^\infty |f(t)|^n \, dt \right)^{1/n}.
\]
Now, for $1 \leq n < \infty$ and $0 < \alpha \leq 1$, we define the weighted space
\[
L^n_\alpha(0,\infty) = \left\{ f:(0,\infty)\to \mathbb{R} \;\Big|\; \int_0^\infty |f(t)|^n t^{\alpha-1} \, dt < \infty \right\}.
\]
The corresponding norm is given by
\[
\|f\|_{L^n_\alpha} = \left( \int_0^\infty |f(t)|^n t^{\alpha-1} \, dt \right)^{1/n}.
\]
Special case $n=1$:
\[
\|f\|_{L^1_\alpha} = \int_0^\infty |f(t)| \, t^{\alpha-1} \, dt.
\]
\begin{theorem}[Boundedness in $L^1_\alpha$]
If $f, g \in L^1_\alpha(0,\infty)$, then
\[
\|f *_\alpha g\|_{L^1_\alpha} \leq \|f\|_{L^1_\alpha} \|g\|_{L^1_\alpha}.
\]
\end{theorem}
\begin{proof}
The comfortable convolution:
$
(f *_\alpha g)(t) = \int_0^t f(p) \, g\!\left((t^\alpha - p^\alpha)^{1/\alpha}\right) p^{\alpha-1} \, dp.
$
We write the $L^1_\alpha$ norm
\[
\|f *_\alpha g\|_{L^1_\alpha} = \int_0^\infty |(f *_\alpha g)(t)| \, t^{\alpha-1} \, dt.
\]
Substitute the convolution definition
\[
\|f *_\alpha g\|_{L^1_\alpha} = \int_0^\infty \left| \int_0^t f(p) \, g\!\left((t^\alpha - p^\alpha)^{1/\alpha}\right) p^{\alpha-1} \, dp \right| t^{\alpha-1} \, dt.
\]
Now apply the Triangle Inequality
\[
\|f *_\alpha g\|_{L^1_\alpha} \leq \int_0^\infty \int_0^t |f(p)| \, \left| g\!\left((t^\alpha - p^\alpha)^{1/\alpha}\right) \right| p^{\alpha-1} \, dp \, t^{\alpha-1} \, dt.
\]
Since the integrand is nonnegative, we may interchange the order of integration:
\[
\|f *_\alpha g\|_{L^1_\alpha} \leq \int_0^\infty |f(p)| \, p^{\alpha-1} \left( \int_p^\infty \left| g\!\left((t^\alpha - p^\alpha)^{1/\alpha}\right) \right| t^{\alpha-1} \, dt \right) dp.
\]
Let
$ u^\alpha = t^\alpha - p^\alpha, \quad \alpha u^{\alpha-1} \, du = \alpha t^{\alpha-1} \, dt, \quad\Rightarrow\quad u^{\alpha-1} \, du = t^{\alpha-1} \, dt.$
Therefore, the inner integral becomes
$
\int_0^\infty |g(u)| \, u^{\alpha-1} \, du = \|g\|_{L^1_\alpha},
$
Substituting this into the above equation, we obtain
\[
\|f *_\alpha g\|_{L^1_\alpha} \leq \int_0^\infty |f(p)| \, p^{\alpha-1} \, dp \cdot \|g\|_{L^1_\alpha}
= \|f\|_{L^1_\alpha} \|g\|_{L^1_\alpha}.
\]
\end{proof}
\begin{theorem}[Young Inequality in $L^n_\alpha$]
    If \(f \in L^1_\alpha(0,\infty)\) and \(g \in L^n_\alpha(0,\infty)\) where $1 \leq n < \infty$ , then 
    $
    \|f *_\alpha g\|_{L^n_\alpha} \leq \|f\|_{L^1_\alpha} \|g\|_{L^n_\alpha}$ and \(f *_\alpha g \in L^n_\alpha(0,\infty)\).
    
\end{theorem}
\textbf{Proof}
The comfortable convolution:
$
(f *_\alpha g)(t) = \int_0^t f(p) \, g\!\left((t^\alpha - p^\alpha)^{1/\alpha}\right) p^{\alpha-1} \, dp.
$
We write the $L^n_\alpha$ norm
\[
\|f *_\alpha g\|_{L^n_\alpha} = \left( \int_0^\infty |(f *_\alpha g)(t)|^n \, t^{\alpha-1}\,dt \right)^{1/n}.
\]
Substitute the convolution definition
\[
\|f *_\alpha g\|_{L^n_\alpha} = \left( \int_0^\infty \left| \int_0^t f(p) \, g\!\left((t^\alpha - p^\alpha)^{1/\alpha}\right) p^{\alpha-1} \, dp \right|^n t^{\alpha-1} \, dt \right)^{1/n}
\]
By Minkowski’s inequality for integrals,
\[
\|f *_\alpha g\|_{L^n_\alpha} \leq \int_0^\infty  \left( |f(p)|^n \, p^{n(\alpha-1)}\int_0^\infty |g((t^\alpha - p^\alpha)^{1/\alpha})|^n \, t^{\alpha-1} dt \right)^{1/n} \, dp.
\]
\[
\|f *_\alpha g\|_{L^n_\alpha} \leq \int_0^\infty |f(p)| \, p^{\alpha-1} \left( \int_0^\infty |g((t^\alpha - p^\alpha)^{1/\alpha})|^n \, t^{\alpha-1} dt \right)^{1/n} \, dp. 
\]
Let
$ u^\alpha = t^\alpha - p^\alpha, \quad \alpha u^{\alpha-1} \, du = \alpha t^{\alpha-1} \, dt, \quad\Rightarrow\quad u^{\alpha-1} \, du = t^{\alpha-1} \, dt.$
Therefore, the inner integral becomes$
\int_0^\infty |g((t^\alpha - p^\alpha)^{1/\alpha})|^n \, t^{\alpha-1} dt = \int_0^\infty |g(u)|^n \, u^{\alpha-1} \, du.
$ Therefore, $
\left( \int_0^\infty |g((t^\alpha - p^\alpha)^{1/\alpha})|^n \, t^{\alpha-1} dt \right)^{1/n} = \|g\|_{L^n_\alpha}
$
Substituting this into the above equation, we obtained
\[
\|f *_\alpha g\|_{L^n_\alpha} \leq \int_0^\infty |f(p)| \, p^{\alpha-1} \, \|g\|_{L^n_\alpha} \, dp= \|g\|_{L^n_\alpha} \int_0^\infty |f(p)| \, p^{\alpha-1} \, dp = \|f\|_{L^1_\alpha} \|g\|_{L^n_\alpha}.
\]
Since $f \in L_{\alpha}^{1}(0,\infty)$ and $g \in L_{\alpha}^{n}(0,\infty)$, their norms are finite. Hence, the above inequality implies that 
\(
\|f *_{\alpha} g\|_{L_{\alpha}^{n}} < \infty .
\) 
Therefore, 
\(
f *_{\alpha} g \in L_{\alpha}^{n}(0,\infty).
\)
\begin{theorem}
     Let \( f, g : [0, \infty) \to \mathbb{R} \) be functions such that their fractional Laplace transforms exist. Then
\[
\mathcal{L}_0^{\alpha}\{f(t)\}(s) \; \mathcal{L}_0^{\alpha}\{g(t)\}(s) = \mathcal{L}_0^{\alpha}\{(f *_{\alpha} g)(t)\}(s), 
\]
where \( *_{\alpha} \) denotes the fractional convolution.
\end{theorem}
\begin{proof}
From the relationship between the conformable and classical Laplace transforms, we write  
\[ \mathcal{L}_0^{\alpha} \{ f(t) \} = \mathcal{L}(f \circ \phi_0^{-1})(s) \quad\text{and}\quad\mathcal{L}_0^{\alpha} \{ g(t) \} = \mathcal{L}(g \circ \phi_0^{-1})(s) \quad\text{where}\quad \phi_0(t) = \frac{t^\alpha}{\alpha}.
\]
Let
\[
\tilde{f}(u) = f(\phi_0^{-1}(u)) \quad\text{and} \quad \tilde{g}(u) = g(\phi_0^{-1}(u)),\quad  u = \phi_0(t).
\]
Then
\[
\mathcal{L}_0^{\alpha} \{ f(t) \}(s) = \mathcal{L} \{ \tilde{f}(u) \}, \quad \mathcal{L}_0^{\alpha} \{ g(t) \}(s) = \mathcal{L} \{ \tilde{g}(u) \}
\]
Therefore, 
\[
\mathcal{L}_0^{\alpha}\{f(t)\}(s) \; \mathcal{L}_0^{\alpha}\{g(t)\}(s) = \mathcal{L}\{\tilde{f}(u)\} (s)\; \mathcal{L}\{\tilde{g}(u)\}(s).
\]
Now using the classical convolution theorem,
\begin{align*}
   \mathcal{L}\{\tilde{f}(u)\}(s) \; \mathcal{L}\{\tilde{g}(u)\}(s)
   &= \mathcal{L}\{(\tilde{f} * \tilde{g})(u)\}(s) \\
   &= \int_{0}^{\infty} e^{-su} \left(\int_0^u \tilde{f}(y) \, \tilde{g}(u - y) \, dy \right) du.
\end{align*}
Since
\(
(\tilde{f} * \tilde{g})(u) = \int_0^u \tilde{f}(y) \, \tilde{g}(u - y) \, dy .
\)
Substituting,
\(
\tilde{f}(u) = f\!\left((\alpha u)^{1/\alpha}\right)
\), the inner convolution becomes
\[
\mathcal{L}\{(\tilde{f} * \tilde{g})(u)\}(s) = \int_{0}^{\infty} e^{-su}\left[ \int_0^u f\!\left((\alpha y)^{1/\alpha}\right)\;g\!\left((\alpha(u- y))^{1/\alpha}\right)dy \right] du 
\]
Now introduce the change of variables
\(
u = \frac{ t^{\alpha}}{\alpha}, y = \frac{ p^{\alpha}}{\alpha}\)  implis \( dy = p^{\alpha-1} dp. 
\)
Moreover,
\(
u - y = \frac{t^{\alpha} - p^{\alpha}}{\alpha}.
\)
Substituting this into the classical convolution formula and transforming back to the variable \(t\), we obtain the fractional convolution:
\begin{align*}
\mathcal{L}\{(\tilde{f} * \tilde{g})(u)\}(s)
& =\int_{0}^{\infty} e^{-s\frac{ t^{\alpha}}{\alpha}}
\left[ \int_{0}^{t} f(p) \; g\!\left( \big( t^{\alpha} - p^{\alpha}\big)^{1/\alpha} \right)
p^{\alpha-1} \, dp \right] dt \\
&=\mathcal{L}_0^{\alpha}\{(f *_{\alpha} g)(t)\}(s).
\end{align*}
\end{proof}
\section{Applications to Fractional Initial Boundary Value Problems}
With the main theoretical findings of the conformable fractional Laplace transform having been made, we now turn our attention to its practical implementation. Specifically, the transform method is very helpful in solving the fractional initial-boundary value problem that arises in transport and diffusion phenomena. The examples given below demonstrate how the developed framework can be utilized to obtain analytical solutions.
\begin{example}[Conformable Fractional First-Order IBVP]Consider the classical first-order initial-boundary value problem
$ 
u_t + x u_x = x, \quad x > 0, \; t > 0, 
$
with the initial and boundary conditions 
$ u(x,0) = 0 \quad \text{for } x > 0,\quad u(0,t) = 0 \quad \text{for } t > 0.
$
This classical formulation is extended to the conformable fractional setting by replacing the first-order time derivative with the conformable fractional derivative of order \(0 < \alpha \le 1\).
\[
    T_t^\alpha u(x,t) + x u_x(x,t) = x, \quad x > 0, \; t > 0 
\]
with the initial and boundary conditions
\[
    u(x,0) = 0, \qquad u(0,t) = 0.
\]
Applying conformable Laplace transform $\mathcal{L}_{\alpha}$ to equation
\[
   \mathcal{L}_{\alpha}\{T_t^\alpha u(x,t)\} +  \mathcal{L}_{\alpha}\{x u_x (x,t)\}= \mathcal{L}_{\alpha}\{x\} 
\]
\[
   sU(x,s) - u(x,0) + x \frac{dU(x,s)}{dx} = \frac{x}{s} 
\]
Since \( u(x,0) = 0 \);
\[
    s U(x,s) + x \frac{dU}{dx} = \frac{x}{s} \quad\Rightarrow \quad
xU_x + sU = \frac{x}{s} \quad\Rightarrow \quad   U_x + \frac{s}{x} U = \frac{1}{s}
\]
This is a first-order linear ODE of the form
$
U' + P(x)U = Q(x)
$
with
\[
  P(x) = \frac{s}{x}, \quad Q(x) = \frac{1}{s}.  
\]
Integrating factor $ x^{s}$,  the solution of this transformed equation is
\[
    U(x,s) = \frac{x}{s(s+1)} + C x^{-s}
\]
where \( C \) is a constant of integration. Since \( U(0,s) = 0 \), \( C = 0 \) for a bounded solution. Consequently
\[
   U(x,s) = \frac{x}{s(s+1)}= x\left( \frac{1}{s} - \frac{1}{s+1} \right)
\]
Using the conformable Laplace transform:
\[
    \mathcal{L}_{\alpha}^{-1}\!\left( \frac{1}{s} \right) = 1, \quad
\mathcal{L}_{\alpha}^{-1}\!\left( \frac{1}{s+1} \right) = e^{-t^{\alpha}/\alpha}
\]
Thus,
\[
    u(x,t) = x \left( 1 - e^{-t^{\alpha}/\alpha} \right).
\]
\end{example}
\begin{example}[Time-Fractional Diffusion Equation in a Semi-Infinite Medium] Consider the conformable diffusion equation in a semi-infinite medium
\begin{equation}
 T_t^\alpha u(x,t) = \kappa \, u_{xx}(x,t), \quad x > 0, \; t > 0, \; 0 < \alpha \leq 1  
 \label{eq:6.1}
\end{equation}
subject to the initial and boundary conditions
\begin{equation}
   u(x,0) = 0, \quad x > 0,  
   \label{eq:6.2}
\end{equation}
\begin{equation}
    u(0,t) = f(t), \quad t > 0,
    \label{eq:6.3}
\end{equation}
\begin{equation}
   u(x,t) \to 0 \quad \text{as} \quad x \to \infty.  
   \label{eq:6.4}
\end{equation}
Here \( T_t^\alpha \) denotes the conformable fractional derivative of order \( \alpha \) with respect to time.\\
Applying conformable Laplace transform $\mathcal{L}_{\alpha}$ to equation \eqref{eq:6.1}
\[
  \mathcal{L}_{\alpha}\{T_t^\alpha u\} = \mathcal{L}_{\alpha}\{k\frac{\partial^2 u}{\partial x^2}\}  
\]
we obtain
\[
    sU(x,s) - u(x,0)=k\frac{\partial^2 U(x,s)}{\partial x^2}
\]
Using the initial condition \eqref{eq:6.2}, we obtain 
$sU(x,s) = \kappa U_{xx}(x,s).$
Rewriting,
\begin{equation}
  U_{xx}(x,s) - \frac{s}{\kappa} U(x,s) = 0. 
  \label{eq:6.5}
\end{equation}
Equation \eqref{eq:6.5} is a second-order ordinary differential equation in 
$x.$
The general solution of \eqref{eq:6.5}  is
\[
    U(x,s) = A(s) e^{-x\sqrt{s/\kappa}} + B(s) e^{x\sqrt{s/\kappa}}. 
\]
The boundedness condition \eqref{eq:6.4} implies \( B(s) = 0 \).
Using the boundary condition \eqref{eq:6.3},
\[
    U(0,s) = \mathcal{L}_{\alpha}\{f(t)\}(s) = F_{\alpha}(s),
\]
we obtain
\[
    U(x,s) = F_{\alpha}(s) \, e^{-x\sqrt{s/\kappa}}. 
\]
Using the classical Laplace transform,
\[
    \mathcal{L}^{-1}\!\left\{ \frac{e^{-a\sqrt{s}}}{\sqrt{s}} \right\}
= \frac{1}{\sqrt{\pi u}} \exp\!\left( -\frac{a^2}{4u} \right). 
\]
With \( a = x/\sqrt{\kappa} \) and applying the conformable inversion theorem, we obtain
\begin{equation}
    g(t) = \frac{x}{2\sqrt{\pi \kappa}} \left( \frac{t^{\alpha}}{\alpha} \right)^{-3/2}
\exp\!\left( -\frac{x^2}{4\kappa (t^{\alpha}/\alpha)} \right).
\label{eq:6.6}
\end{equation}
Now, by using fractional convolution
\[
   \mathcal{L}_{\alpha}^{-1}\{U(x,s)\}=\mathcal{L}_{\alpha}^{-1}\{F_{\alpha}(s) G(s)\} = (f *_{\alpha} g)(t) 
\]
\begin{equation}
   u(x,t)=(f *_{\alpha} g)(t) = \int_{0}^{t} f(p) \, g\!\left( (t^{\alpha} - p^{\alpha})^{1/\alpha} \right) p^{\alpha-1} \, dp 
   \label{eq:6.7}
\end{equation}
Substituting \( g(t) \) from \eqref{eq:6.6} into the fractional convolution \eqref{eq:6.7}, we obtain a solution of the problem \eqref{eq:6.1}--\eqref{eq:6.4} is
\[
    u(x,t) = \frac{x\alpha^{3/2}}{2\sqrt{\pi \kappa}} \int_{0}^{t} f(p) \, (t^{\alpha} - p^{\alpha})^{-3/2}
\exp\!\left( -\frac{\alpha x^{2}}{4\kappa (t^{\alpha} - p^{\alpha})} \right) p^{\alpha-1} \, dp.
\]
which is, by putting 
$\lambda = \frac{x \sqrt{\alpha}}{2 \sqrt{\kappa (t^{\alpha} - p^{\alpha})}}$ and $p^{\alpha-1} dp = \frac{x^{2}}{2 \kappa} \lambda^{-3} d\lambda$,
\[
    u(x,t) = \frac{2}{\sqrt{\pi}} \int_{\frac{x\sqrt{\alpha}}{2\sqrt{\kappa t^{\alpha}}}}^{\infty} f\!\left( \left( t^{\alpha} - \frac{\alpha x^{2}}{4\kappa \lambda^{2}} \right)^{1/\alpha} \right) e^{-\lambda^{2}} \, d\lambda.
\]
This is the formal solution of the problem.
\end{example}
\begin{example}[Time-Fractional Diffusion Equation in finite Medium]
 Consider the one-dimensional conformable time-fractional diffusion equation in a finite medium:
\begin{equation}
T_\alpha u(x,t) = \kappa \frac{\partial^2 u(x,t)}{\partial x^2}, \quad 0 < x < a, \; t > 0,
\label{eq:6.8}
\end{equation}
where \(0 < \alpha \leq 1\), \(\kappa > 0\) is the diffusion coefficient, and \(T_\alpha\) denotes the conformable fractional derivative of order \(\alpha\) .
The system is subject to the initial condition
\begin{equation}
u(x,0) = 0, \quad 0 < x < a,
\end{equation}
and the mixed boundary conditions
\begin{equation}
    u(0,t) = U, \quad t > 0, 
    \label{eq:6.10}
\end{equation}
\begin{equation}
    \frac{\partial u}{\partial x}(a,t) = 0, \quad t > 0,
\label{eq:6.11}
\end{equation}
where \(U\) is a prescribed constant.\\
Apply the Conformable Laplace Transform in \(t\) of the equation \eqref{eq:6.8}
\[
    \mathcal{L}_\alpha\{T_t^\alpha u\} = s\bar{u}(x,s) - u(x,0)
\]
Since \(u(x,0) = 0\),
\(
s\bar{u}(x,s) = \kappa \frac{d^2\bar{u}}{dx^2}
\).
Hence,
\(
\frac{d^2\bar{u}}{dx^2} - \frac{s}{\kappa}\bar{u} = 0
\). This is the same spatial ODE as the classical case.
So general solution
\[
    \bar{u}(x,s) = A \cosh\left(x \sqrt{\frac{s}{\kappa}}\right) + B \sinh\left(x \sqrt{\frac{s}{\kappa}}\right)
\]
Apply Boundary Conditions \eqref{eq:6.10}, \(u(0,t) = U\).
Taking the conformable Laplace transform:
\(\bar{u}(0,s) = \frac{U}{s}\), thus
\(A = \frac{U}{s}\).
Apply Boundary Conditions \eqref{eq:6.11}, \(u_x(a,t) = 0\).\\
Differentiate:
\[
    \bar{u}_x = A \sqrt{\frac{s}{\kappa}} \sinh\left(x \sqrt{\frac{s}{\kappa}}\right) + B \sqrt{\frac{s}{\kappa}} \cosh\left(x \sqrt{\frac{s}{\kappa}}\right)
\]
At \(x=a:\)
\[
    A \sinh\left(a \sqrt{\frac{s}{\kappa}}\right) + B \cosh\left(a \sqrt{\frac{s}{\kappa}}\right) = 0
\]
Solve for \(B\)
\[
    B = -A \tanh\left(a \sqrt{\frac{s}{\kappa}}\right)
\]
Substitute \(A = \frac{U}{s}\).
After simplification:
\[
  \bar{u}(x,s) = \frac{U}{s} \frac{\cosh\left((a-x) \sqrt{\frac{s}{\kappa}}\right)}{\cosh\left(a \sqrt{\frac{s}{\kappa}}\right)}  
\]
The Cauchy Residue Theorem can carry out the inversion 
\[
   u(x,t) = \sum_k \text{Res}\left(e^{s\frac{t^\alpha}{\alpha}} \bar{u}(x,s), s = z_k\right) 
\]
Poles occur when,
\(
\cosh\left(a \sqrt{\frac{s}{\kappa}}\right) = 0
\), this gives
\(
\sqrt{\frac{s}{\kappa}} = i \frac{(2n-1)\pi}{2a}
\).
So
\(
s_n = -\kappa \left(\frac{(2n-1)\pi}{2a}\right)^2
\)
Using the same residue computation as the classical case to obtain
\[
    u(x,t) = U \left[ 1 + \frac{4}{\pi} \sum_{n=1}^{\infty} \frac{(-1)^n}{2n-1} \cos\left( \frac{(2n-1)\pi}{2a}(a-x) \right) \exp\left( -\kappa \left( \frac{(2n-1)\pi}{2a} \right)^{2\alpha} t^\alpha \right) \right]
\]
\[
    u(x,t) = U \left[ 1 - \frac{4}{\pi} \sum_{n=1}^{\infty} \frac{1}{2n-1} \sin\left(\frac{(2n-1)\pi x}{2a}\right) \exp\left(-\kappa \left(\frac{(2n-1)\pi}{2a}\right)^2 \frac{t^{\alpha}}{\alpha}\right) \right].   
\]
Similar to the classical diffusion equation, the above solution may also be obtained by the method of separation of variables in the conformable fractional framework.
\end{example}
\begin{example}
    Consider a finite domain problem with Dirichlet boundary conditions
\begin{equation}
    T_t^\alpha u(x,t) = \frac{\partial^2 u(x,t)}{\partial x^2}, \quad 0 < x < \pi,\; t > 0, 
    \label{eq:6.12}
    \end{equation}
with initial and boundary conditions
\begin{equation}
    u(x,0) = \sin x, \quad u(0,t) = u(\pi,t) = 0
    \label{eq:6.13}
\end{equation}
Applying the conformable fractional Laplace transform on the equation \eqref{eq:6.12} 
\[
    \mathcal{L}_{\alpha}\{T_t^\alpha u\} = \mathcal{L}_{\alpha}\{\frac{\partial^2 u}{\partial x^2}\}
\]
we obtain
\[
    sU(x,s) - u(x,0)=\frac{\partial^2 U(x,s)}{\partial x^2}
\]
Using \eqref{eq:6.13},
\[
    sU(x,s) - \sin x = \frac{\partial^2 U(x,s)}{\partial x^2} 
\]
\begin{equation}
    \frac{\partial^2 U}{\partial x^2}- sU = - \sin x
    \label{eq:6.14}
\end{equation}
This is a linear second-order equation in x. \\
Assume
\begin{equation}
    U(x,s) = A(s)\sin x,
    \label{eq:6.15}
\end{equation}
Since
\(
U(0,s) = U(\pi,s) = 0
\).
Differentiating equation \eqref{eq:6.15} twice, we obtain
\begin{equation}
    \frac{\partial^2 U}{\partial x^2}= - A(s)\sin x
    \label{eq:6.16}
\end{equation}
Putting \eqref{eq:6.15} and \eqref{eq:6.16} in \eqref{eq:6.14}, we obtain
\[
- A(s)\sin x- A(s)\sin x = - \sin x, \quad \Rightarrow \quad A(s)\sin x(s+1) = \sin x
\]
\begin{equation}
    A(s) = \frac{1}{s + 1}
    \label{eq:6.17}
\end{equation}
Putting \eqref{eq:6.17} in \eqref{eq:6.15}, we get 
\[
    U(x,s) = \frac{\sin x}{s + 1}
\]
Now apply the conformable fractional inversion theorem 
\[
    \mathcal{L}_{\alpha}^{-1}\{U(x,s)\}= \mathcal{L}_{\alpha}^{-1}\!\left\{\frac{\sin x}{s+1}\right\}
\]
\[
    u(x,t) = \sin x \; \mathcal{L}^{-1}_{\alpha}\!\left\{\frac{1}{s+1}\right\}
\]
\[
    u(x,t) = \sin x \; e^{-\frac{t^{\alpha}}{\alpha}}.
\]
This provides an explicit solution of the conformable time-fractional diffusion equation with Dirichlet boundary conditions.
\end{example}
\textbf{Remark:}
(1)  We can use any other diffeomorphism from $[0,\infty)$ to $[0,\infty)$ instead of 
$\phi_a(t)=\frac{(t-a)^\alpha}{\alpha}$ to widen the class of functions in applications.\\
(2) Also, if we use the Henstock--Kurzweil (HK) integration instead of the Lebesgue 
integration in the definition of the conformable Laplace transform, then it becomes 
applicable to more oscillatory functions.

\section{Conclusion}
This paper presents a systematic study of the Conformable Fractional Laplace Transform (CFLT) and its application to fractional diffusion equations. The fundamental properties of the transform, including its relationship with the classical Laplace transform, have been examined in detail. In particular, the fractional convolution operator associated with the conformable Laplace transform was analyzed, and important properties such as commutativity, associativity, and distributivity were established. These results confirm that the conformable convolution maintains the essential algebraic structure of the classical convolution framework.
Furthermore, the boundedness of the fractional convolution operator in the weighted space $L_{\alpha}^{1}$ and the corresponding Young-type inequality in $L_{\alpha}^{n}$ were demonstrated. These theoretical results provide a rigorous functional-analytic foundation for applying the conformable Laplace transform in solving fractional differential equations.\\
The effectiveness of the proposed framework was illustrated through several applications to fractional initial-boundary value problems and diffusion equations in both semi-infinite and finite domains. By applying the conformable Laplace transform method, explicit analytical solutions were obtained for these problems. The results show that the conformable fractional model preserves the structural simplicity of the classical Laplace transform while extending its applicability to fractional-order systems.\\
Overall, the study confirms that the conformable fractional Laplace transform is a powerful analytical tool for solving fractional diffusion equations and related problems. The developed methodology not only simplifies the solution process but also provides a consistent bridge between classical and fractional models. This framework may be further extended to investigate more complex fractional partial differential equations and other applied problems arising in physics, engineering, and applied mathematics.
\bibliographystyle{plain}

\end{document}